\documentclass[a4paper,11pt]{article}
\usepackage[utf8]{inputenc}

\usepackage{amsthm,amsmath,amsfonts,amssymb}
\usepackage[english]{babel}
\usepackage{graphicx}
\usepackage{verbatim}

\DeclareMathOperator{\diag}{diag}
\DeclareMathOperator{\dist}{dist}

\DeclareMathOperator{\diam}{diam}

\def\u{\mbox{\boldmath $u$}}

\def\vecu{\mbox{\boldmath $u$}}
\def\vecv{\mbox{\boldmath $v$}}
\def\vecw{\mbox{\boldmath $w$}}

\def\vec0{\mbox{\boldmath $0$}}

\def\A{\mbox{\boldmath $A$}}

\def\A{\mbox{\boldmath $A$}}

\def\D{\mbox{\boldmath $D$}}

\def\I{\mbox{\boldmath $I$}}

\def\M{\mbox{\boldmath $M$}}

\def\V{\mbox{\boldmath $V$}}

\def\I{\mbox{\boldmath $I$}}

\def\M{\mbox{\boldmath $M$}}

\theoremstyle{plain}   % Cal carregar el paquet theorem.sty o amsthm.sty

\newtheorem{theorem}{Theorem}[section]
\newtheorem{proposition}[theorem]{Proposition}
\newtheorem{corollary}[theorem]{Corollary}

\newtheorem{definition}[theorem]{Definition}

\title{Sequence mixed graphs}
\author{C. Dalf\'o$^a$, M. A. Fiol$^b$, N. L\'opez$^c$
\\ \\
{\small $^{a,b}$Dep. de Matem\`atiques, Universitat Polit\`ecnica de Catalunya} \\
{\small $^b$Barcelona Graduate School of Mathematics} \\
{\small 08034 Barcelona, Catalonia} \\
{\small {\tt\{cristina.dalfo,miguel.angel.fiol\}@upc.edu}} \\
{\small $^c$Dep. de Matem\`atica, Universitat de Lleida}\\
 {\small 25001 Lleida, Catalonia}\\
 {\small {\tt nlopez@matematica.udl.es}}
}
\date{}

\begin{document}

%\tableofcontents
\maketitle

\begin{abstract}
A mixed graph can be seen as a type of digraph containing some edges (two opposite arcs). Here  we introduce the concept of sequence mixed graphs, which is a generalization of both sequence graphs and iterated line digraphs. These structures are proven to be useful  in the problem of constructing dense graphs or digraphs, and this is related to the degree/diameter problem. Thus, our generalized approach gives rise to graphs that have also good ratio order/diameter. Moreover, we propose a general method for obtaining a sequence mixed digraph by identifying some vertices of a certain iterated line digraph. As a consequence, some results about distance-related parameters (mainly, the diameter and the average distance) of sequence mixed graphs are presented.
\end{abstract}


\section{Introduction}
Two techniques that have proved very useful to obtain large graphs and digraphs are, respectively, sequence graph and line digraph approaches. Sequence graphs were first proposed by  Fiol, Yebra, and F\`abrega \cite{fyf83}, whereas, within this context, line digraphs were studied by Fiol, Yebra, and Alegre \cite{FiYeAl83,FiYeAl84}. Mixed graphs can be seen as a generalization of both, undirected and directed graphs, see for instance the works by Nguyen and Miller \cite{nm08} and Buset, El Amiri, Erskine, Miller and P\'erez-Ros\'es \cite{baemp15}.
Here we introduce the concept of sequence mixed graphs, which is a generalization of both sequence graphs and iterated line digraphs. In particular, we show that a sequence mixed graph can be obtained by identifying some vertices of a certain iterated line digraph. This allows to apply known results of the latter to study some basic distance-related properties in the introduced structures. First, we begin with some standard definitions.

A {\em mixed} (or {\em partially directed\/}) graph $G$ with vertex set $V$ may contain (undirected) {\em edges} as well as directed edges (also known as {\em arcs}). From this point of view, a {\em graph} [resp. {\em directed graph} or {\em digraph}] has all its edges undirected [resp. directed]. In fact, we can identify the mixed graph $G$ with its associated digraph $G^*$ obtained by replacing all the edges by digons (two opposite arcs or a directed $2$-cycle). The {\em undirected degree} of a vertex $v$, denoted by $d(v)$ is the number of edges incident to $v$. The {\em out-degree} [resp. {\em in-degree}] of vertex $v$, denoted by $d^+(v)$ [resp. $d^-(v)$], is the number of arcs emanating from [resp. to] $v$.  If $d^+(v)=d^-(v)=z$ and $d(v)=r$, for all $v \in V$, then $G$ is said to be {\em totally regular\/} of degree $(r,z)$ (or simply {\em $(r,z)$-regular}).
Note that, in this case, the corresponding digraph $G^*$ is $(r+z)$-regular. A {\em walk\/} of length $\ell\geq 0$ from $u$ to $v$ is a sequence of $\ell+1$ vertices, $u_0u_1\dots u_{\ell-1}u_\ell$, such that $u=u_0$, $v=u_\ell$ and each pair $u_{i-1}u_i$, for $i=1,\ldots,\ell$, is either an edge or an arc of $G$. A {\em directed walk} is a walk containing only arcs. An {\em undirected walk} is a walk containing only edges. A walk whose vertices are all different is called a {\em path}.
%A {\em directed path} is a path containing only arcs.
The length of a shortest path from $u$ to $v$ is the {\it distance\/} from $u$ to $v$, and it is denoted by $\dist(u,v)$. Note that $\dist(u,v)$ may be different from $\dist(v,u)$, when shortest paths between $u$ and $v$ involve arcs. The maximum distance between any pair of vertices is the {\it diameter} $k$ of $G$, while the {\em average distance} between vertices of $G$ is defined as
$$
\overline{k}=\frac{1}{|V|^2}\sum_{u,v\in V}\dist(u,v).
$$
A {\em directed cycle} [resp. {\em undirected cycle}] of length $\ell$ is a walk of length $\ell$ from $u$ to $v$ involving only arcs [resp. edges] whose vertices are all different except $u=v$.
Finally, notice that the adjacency matrix $\A=(a_{uv})$ of a mixed graph $G$ coincides with the adjacency matrix of its associated digraph $G^*$, where $a_{uv}=1$ if there is an arc from $u$ to $v$, and $a_{uv}=0$ otherwise.

The following concepts, sequence graphs, line digraphs, and related results were studied in the works by Alegre, F\`abrega, Fiol, and Yebra \cite{fyf83,FiYeAl83,FiYeAl84}.

\begin{definition}
\label{def:seq-graphs}
Given a digraph $G$, each vertex of its {\em $\ell$-iterated line digraph} $L^{\ell}(G)$ represents a walk $u_0u_1\dots u_{\ell-1}u_\ell$ of length $\ell$ in $G$, and  vertex $u=u_0u_1\dots u_{\ell-1}u_{\ell}$ is adjacent to the vertices of the form $v=u_1u_2\dots u_{\ell}u_{\ell+1}$ with $(u_{\ell},u_{\ell+1})$ being an arc of $G$.
\end{definition}

%Intuitively speaking, in $L^{\ell}(G)$ one walk (of $G$) is adjacent to another one if there exists a `walk-displacement' of length $1$ that overlaps the first sequence onto the second one.
By way of example, in Figure \ref{fig:S3G-L3G*}$(b)$ we represent the (symmetric) digraph $G^*$  and its  $3$-iterated line digraph $L^3(G^*)$.
The following result shows that the line digraph technique is useful to obtain dense digraphs.

\begin{theorem}\mbox{\rm \cite{FiYeAl83,FiYeAl84}}
\label{th:diam-line}
Let $G$ be a regular digraph different from a directed cycle, with diameter $k$ and average distance between vertices $\overline{k}$. Then, the diameter $k_{\ell}$ and average distance $\overline{k}_{\ell}$ of $L^{\ell}(G)$ satisfy
\begin{align}
k_{\ell} &= k+\ell, \label{(1)}\\
\overline{k}_{\ell} &< \overline{k}+\ell.
\end{align}
Moreover, if $G$ is nonregular, then \eqref{(1)} also holds.
\end{theorem}

\begin{definition}
\label{def:seq-graphs}
Given a graph $G$, the vertices of the {\em sequence graph} $S^{\ell}(G)$ of $G$ (also known as $\ell$-sequence graph) are all the walks $u_0u_1\dots u_{\ell-1}u_\ell$ of length $\ell$ of $G$. The edges are defined as follows: vertex $u_0u_1\dots u_{\ell-1}u_{\ell}$ is adjacent to $vu_0u_1\dots u_{\ell-1}$ and $u_0u_1\dots u_{\ell-1}w$ with $(v,u_0)$ and $(u_{\ell-1},w)$ being edges of $G$.
\end{definition}

%\begin{figure}
%  % Requires \usepackage{graphicx}
%  \includegraphics[width=]{}\\
%  \caption{}\label{}
%\end{figure}
\noindent Within this definition, we consider a walk $u_0u_1\dots u_{\ell
-1}u_{\ell}$ and its {\em conjugate} $u_{\ell}u_{\ell-1}\dots u_1u_0$ as the same sequence (or walk). Moreover, as we only consider simple graphs, self-adjacencies (or loops) are not taken into consideration. As an example, Figure \ref{fig:S3G-L3G*}$(a)$ shows the graph $G=K_{2,2}$  and its  $3$-sequence graph $S^3(G)$. In our context of construction of dense graphs, the main property of sequence graphs is the following:

\begin{theorem}\mbox{\rm \cite{fyf83}}
\label{th:diam}
Let $G$ be a graph of diameter $k$. Then, the diameter $k_{\ell}$ of $S^{\ell}(G)$ satisfies
$$
k_{\ell}\le k+\ell.
$$
\end{theorem}

\begin{figure}[t]
\begin{center}
\vskip-.75cm
\includegraphics[width=12cm]{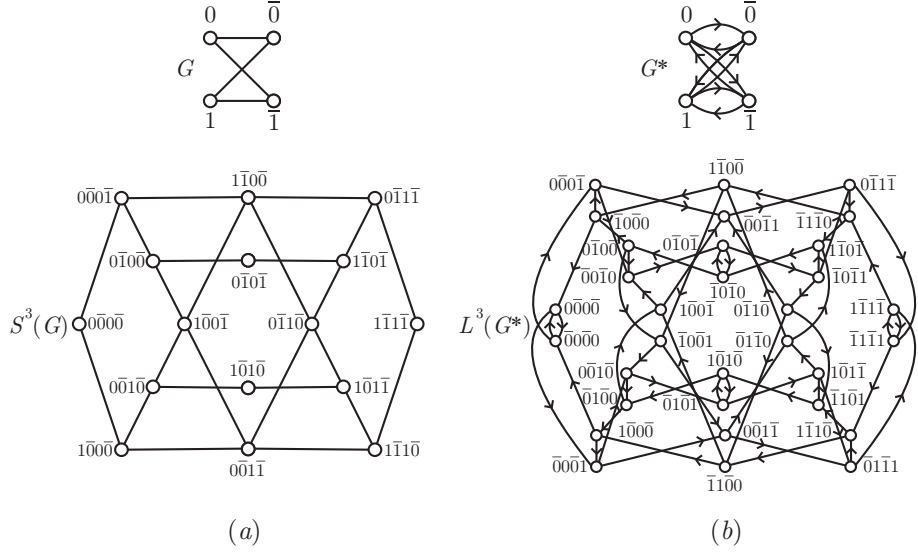}
\vskip-8cm
\caption{$(a)$ A graph $G$ and its  3-sequence graph $S^3(G)$;
$(b)$ The symmetric digraph $G^*$ and its 3-iterated line digraph $L^3(G^*)$.}
\label{fig:S3G-L3G*}
\end{center}
\end{figure}

\section{The Degree/diameter problem for mixed graphs}

The degree/diameter problem for mixed graphs asks for the largest possible number of vertices $n(r,z,k)$ in a mixed graph with maximum undirected degree $r$, maximum directed out-degree $z$, and diameter $k$. Most of the main results regarding this problem for undirected and directed graphs can be found in the comprehensive survey of Miller and \v{S}ir\'a\v{n} \cite{ms13}. Nevertheless, little is known about this extremal problem on mixed graphs. A natural upper bound for $n(r,z,k)$ is derived just by counting the number of vertices at every distance from any given vertex $v$ in a mixed graph with given maximum undirected degree $r$, maximum directed out-degree $z$, and diameter $k$. This bound is known as the {\em Moore bound} for mixed graphs (see Buset, El Amiri, Erskine, Miller, and P\'{e}rez-Ros\'{e}s~\cite{baemp15}):
\begin{equation}
\label{eq:moorebound4}
% \begin{split}
M(r,z,k)=A\frac{u_1^{k+1}-1}{u_1-1}+B\frac{u_2^{k+1}-1}{u_2-1},
% \end{split}
\end{equation}
where
%$v=(z+r)^2+2(z-r)+1$ and
%
%\begin{tabular}{ll}
%$u_1=\frac{1}{2}(z+r-1-\sqrt{v})$, & $A=\frac{1}{2\sqrt{v}}\big(\sqrt{v}-(z+r+1)\big)$, \\
%$u_2=\frac{1}{2}(z+r-1+\sqrt{v})$, & $B=\frac{1}{2\sqrt{v}}\big(\sqrt{v}+(z+r+1)\big)$. \\
%\end{tabular}
\begin{align}
v   &=(z+r)^2+2(z-r)+1, \label{v} \\
u_1 &=\frac{1}{2}(z+r-1-\sqrt{v}),\qquad u_2=\frac{1}{2}(z+r-1+\sqrt{v}),\label{u's} \\
A &=\frac{\sqrt{v}-(z+r+1)}{2\sqrt{v}},
\qquad
B =\displaystyle{\frac{\sqrt{v}+(z+r+1)}{2\sqrt{v}}}. \nonumber
\end{align}

Another way to compute the Moore bound for mixed graphs is as follows. Let $G$ be a $(r,z)$-regular mixed graph with $d=r+z$. Given a vertex $v$ and for $i=0,1,\ldots,k$, let $N_i$ be the maximum number of vertices at distance $i$ from $v$. Let $N_i=R_i+Z_i$, where $R_i$ is the number of vertices that, in the corresponding tree rooted at $v$,  have an edge with their parents; and  $Z_i$ is the number of vertices that has an arc from their parents. Then,
$$
N_i = R_i+Z_i = ((r-1)+z)R_{i-1}+(r+z)Z_{i-1}.
$$
That is, $R_i = (r-1)R_{i-1}+rZ_{i-1}$, and $Z_i = zR_{i-1}+zZ_{i-1}$ or, in matrix form,
$$
\left(
\begin{array}{c}
  R_i \\
  Z_i
\end{array}
\right)=
\left(
\begin{array}{cc}
  r-1 & r\\
  z   & z
\end{array}
\right)
\left(
\begin{array}{c}
  R_{i-1} \\
  Z_{i-1}
\end{array}
\right)=\cdots=\M^i\left(
\begin{array}{c}
  R_{0} \\
  Z_{0}
\end{array}
\right)=
\M^i\left(
\begin{array}{c}
  0 \\
  1
\end{array}
\right),
$$
where $\M=\left(
\begin{array}{cc}
  r-1 & r\\
  z   & z
\end{array}
\right)$, and by convenience $R_0=0$ and $Z_0=1$. Therefore,
$$
N_i = R_i+Z_i =\left(
\begin{array}{cc}
  1 & 1
\end{array}
\right)\M^i\left(
\begin{array}{c}
  0 \\
  1
\end{array}
\right).
$$
Consequently, the Moore bound for mixed graphs is
\begin{align}
  M(r,z,k) &= \sum_{i=0}^{k}N_i=\left(
\begin{array}{cc}
  1 & 1
\end{array}
\right)
\sum_{i=0}^{k}\M^i
\left(
\begin{array}{c}
  0 \\
  1
\end{array}
\right) \nonumber
\\
   &= \left(
\begin{array}{cc}
  1 & 1
\end{array}
\right)
(\M^{k+1}-\I)(\M-\I)^{-1}\left(
\begin{array}{c}
  0 \\
  1
\end{array}
\right) \nonumber
\\
 &=\frac{1}{r+2z-2}\left(
\begin{array}{cc}
  1 & 1
\end{array}
\right)(\M^{k+1}-\I)\left(
\begin{array}{c}
  r \\
  2-r
\end{array}
\right),
\label{MooreBound2}
\end{align}
where we used that
$$
(\M-\I)^{-1}=\frac{1}{r+2z-2}
\left(
\begin{array}{cc}
  1-z & r\\
  z   & 2-r
\end{array}
\right),
$$
with $r+2z\neq2$. Notice that this condition implies that there are two forbidden cases: $r=0$ and $z=1$ that corresponds to a directed cycle, and $r=2$ and $z=0$ that corresponds to an undirected cycle.
Of course, \eqref{MooreBound2} can be shown equivalent to \eqref{eq:moorebound4} since the eigenvalues of  $\M$ are precisely $u_1$ and $u_2$ given in \eqref{u's}. Then, $\M^{k+1}=\V\D^{k+1}\V^{-1}$, where $\D=\diag(u_1,u_2)$, and
$$
\V=\left(\begin{array}{cc}
\frac{r-z-1+\sqrt{v}}{2z} & \frac{r-z-1-\sqrt{v}}{2z}\\
1 & 1
\end{array}
\right),
$$
where $v$ is given by \eqref{v}.

{\em Mixed Moore graphs} are those with order attaining \eqref{eq:moorebound4}, which means that between any pair of vertices there is a unique shortest path of length not greater than the diameter. As a consequence, mixed Moore graphs must be totally regular of degree $(r,z)$ (see Bos\'{a}k~\cite{B79}). Nguyen, Miller, and Gimbert \cite{nmg07} showed that mixed Moore graphs only exist for diameter $k=2$. Although Moore graphs and digraphs are unique when their existence is known, this fact is no longer true for the mixed case. For instance, there exists a unique mixed Moore graph of parameters $r=3$ and $z=1$, known as Bos\'{a}k graph (see again \cite{B79}), but there are at least two nonisomorphic mixed Moore graphs with $r=3$ and $z=7$ (see J{\o}rgensen \cite{J15}). Recently, L\'opez, Miret and Fern\'{a}ndez \cite{LMF2015} proved the nonexistence of mixed Moore graphs of orders $40$, $54$ and $84$, corresponding to the $(r,z)$ pairs $(3,3)$, $(3,4)$ and $(7,2)$. Besides, there are many pairs $(r,z)$ for which
the existence of a mixed Moore graph remains as an open problem. Mixed graphs of diameter $k$ with maximum undirected degree $r$, maximum directed out-degree $z$ and order just one less than the mixed Moore bound are known as {\em mixed almost Moore graphs}. They have been studied only for diameter $k=2$ (see L\'{o}pez and Miret \cite{LM2016}).

\section{Sequence mixed graphs}

One can consider the extension of the definition of sequence graphs for mixed graphs. Since walks may contain arcs, the definition should be slightly modified in the following way.

\begin{definition}
\label{def:gen-seq-graphs}
Given a mixed graph $G$, the vertices of the {\em sequence mixed graph} $S^{\ell}(G)$ of $G$ are all the walks $u_0u_1\dots u_{\ell-1}u_{\ell}$ of length $\ell$ of $G$. The walk $u_0u_1\dots u_{\ell-1}u_{\ell}$ and its conjugate $u_{\ell}u_{\ell-1}\dots u_1u_0$ are the same sequence $($or walk$)$ if and only if this walk is undirected. Moreover, in $S^{\ell}(G)$, vertex $\vecu=u_0u_1\dots u_{\ell-1}u_{\ell}$ has the following adjacencies:
\begin{itemize}
\item
If the walk $u_0u_1\dots u_{\ell-1}u_{\ell}$ is undirected, vertex $\vecu$ is adjacent, through edges, with the vertices $\vecv=u_1u_2\dots  u_{\ell}u_{\ell+1}$   when $(u_{\ell},u_{\ell+1})$ are  edges of $G$ and/or the vertices $\vecv'=u_{0}u_{1}\dots u_{\ell-1}$,  when  $(u_{0},u_{1})$ are  edges of $G$. Here, we assume that $\vecv$ and $\vecv'$ are distinct from $\vecu$ $($loops are not considered\/$)$.
\item
Otherwise, vertex $\vecu$ is adjacent, through arcs, to the vertices $\vecw=u_1u_2\dots  u_{\ell}u_{\ell+1}$, where $(u_{\ell},u_{\ell+1})$ is either an arc or an edge of $G$.
\end{itemize}
\end{definition}

Intuitively speaking, notice that one walk is adjacent to another one if there exists a `walk-displacement' of length one into the mixed graph that overlaps the first sequence onto the second one.
Moreover, $(\u,\vecv)$, but not $(\vecv,\u)$, is an arc of $S^{\ell}(G)$ if $(u_i,u_{i+1})$ is an arc of $G$ for some $i=0,\ldots,\ell$, whereas $(\u,\vecv)$ is an edge (or digon) of $S^{\ell}(G)$ otherwise.

In the case when $G$ is totally regular, the following lemma gives information about the degrees and order of its $\ell$-sequence mixed graph.

\begin{proposition}
\label{prop:reg1}
Let $G$ be an $(r,z)$-regular mixed graph on $n$ vertices. Let $\A$ be the adjacency matrix of the undirected subgraph of $G$. Let $V_1$ be the subset  of $\ell$-walks of $G$ (or vertices of $S^{\ell}(G)$) containing only edges, and let $V_2$ be the subset of $\ell$-walks of $G$ containing at least one arc. Let $S^{\ell}(G)$ have $N_{\ell}(G)$ vertices, maximum undirected degree $\Delta$, and maximum directed in-degree and out-degree $\Delta^*$ of $S^{\ell}(G)$. Then,
\begin{itemize}
\item[$(a)$] The number of vertices in $V_1$ is
$$
|V_1|=\left\{
\begin{array}{ll}
\frac{1}{2}nr^{\ell} & \mbox{if\ \ $\ell$ is odd,}\\[.1cm]
\frac{n}{2}(r^{\ell}+r^{\ell/2}) & \mbox{if\ \ $\ell$ is even,}
\end{array}
\right.
$$
and, for any $\vecv \in V_1$,
\begin{enumerate}
\item
$d(\vecv)=1$, and $d^+(\vecv)=d^-(\vecv)=z$, \ if $r=1$ and $\ell$ is even,
\item
$d(\vecv)=2r-2$, and $d^+(\vecv)=d^-(\vecv)=2z$, \ if $\ell=1$,
\item
$d(\vecv)\le 2r$, and $d^+(\vecv)=d^-(\vecv)\le 2z$, otherwise.
\end{enumerate}
\item[$(b)$] The number of vertices in $V_2$ is
$$
|V_2|=n[(r+z)^{\ell}-r^{\ell}]
$$
and, for any $\vecv \in V_2$,
\begin{enumerate}
\item
$d(\vecv)=0$, and $d^+(\vecv)=d^-(\vecv)=r+z$.
\end{enumerate}
\item[$(c)$] The number of vertices of  $S^{\ell}(G)$ is
$$
N_{\ell}(G)=\left\{
\begin{array}{ll}
n\left[(r+z)^{\ell}+\frac{1}{2}r^{\ell}\right]& \mbox{ \ if\ \ $\ell$ is odd,}\\[.2cm]
n\left[(r+z)^{\ell}+\frac{1}{2}(r^{\ell}+r^{\ell/2})\right] & \mbox{ \ if\ \ $\ell$ is even,}
\end{array}
\right.
$$
and
\begin{enumerate}
\item
$\Delta=1$, and $\Delta^*=z$, \ if $r=1$ and $\ell$ is even,
\item
$\Delta = 2r-2$, and $\Delta^*=\max\{r+z,2z\}$, \ if $\ell=1$,
\item
$\Delta = 2r$, and $\Delta^*=\max\{r+z,2z\}$.
\end{enumerate}
\end{itemize}
\end{proposition}
\begin{proof}
$(a)$ The undirected subgraph of $G$ is $r$-regular and, hence, the number of walks rooted at a given vertex is $r^{\ell}$ for a total of $nr^{\ell}$ in the graph. These  will be double counting vertices in our graph if the walk is self-conjugate (that is, the sequence of vertices defining the walk is palindromic), which is only possible for even $\ell$.
Then, the value of $|V_1|$ results by taking into account that the total number of walks satisfying this condition is $nr^{\ell/2}$. (Indeed, each of such walks is constituted by two equal walks of length $\ell/2$  one going forward and the other backward.) The values of the undirected and directed degrees of every vertex $\vecv\in V_1$ is a simple consequence of Definition \ref{def:gen-seq-graphs}.

$(b)$ Seen as a digraph, $G$ is $(r+z)$-regular. Hence, its total number of rooted $\ell$-walks is $n(r+z)^\ell$. But $nr^{\ell}$ of them consist only of edges (note that, seen as directed walks, there is no double counting). Then, the result follows. Again, the values of the degrees follow from Definition \ref{def:gen-seq-graphs}.

$(c)$ Simply notice that $N_{\ell}(G)=|V_1|+|V_2|$.
\end{proof}

Note that, in some cases, the degrees of $\vecv$ in $(a)3$ can be smaller than the maximum values given there. This precisely happens when $\ell$ is even, say $\ell=2l$, and vertex $\vecv$  is self-conjugate, $\vecv=v_0\ldots v_{l-1}v_{l}v_{l-1}\ldots v_0$, in which case $d(\vecv)=r$ and $d^+(\vecv)=d^-(\vecv)=z$.

In general, as shown in the following result, the $\ell$-sequence of any mixed graph $G$ can be obtained from the $\ell$-iterated line digraph
of its associated digraph $G^*$.
%\subsection*{Sequence graphs vs. line digraphs}

\begin{theorem}
\label{th:sequencenc-line-digraph}
Let $G$ be a mixed graph and $G^*$ its associated digraph. Then, the sequence mixed graph  $S^{\ell}(G)$ can be obtained from the $\ell$-iterated line digraph $L^{\ell}(G^*)$ by identifying each vertex $\u=u_0u_1\dots u_{\ell-1}u_{\ell}$ with its `conjugate' $\overline{\u}=u_{\ell}u_{\ell-1}\dots u_{1}u_0$, if any, and replacing the resulting digons by edges.
\end{theorem}
\begin{proof}
First, notice that every vertex of $L^{\ell}(G^*)$ corresponds to an $\ell$-sequence in $G$ and, hence, to a vertex of $S^{\ell}(G)$. Moreover, the existence of a vertex $\u=u_0u_1\dots u_{\ell-1}u_{\ell}$ and its conjugate $\overline{\u}$ (which can coincide with $\u$) indicates the presence of the undirected walk $u_0u_1\dots u_{\ell-1}u_{\ell}$ and, hence, $\u$ and $\overline{\u}$ correspond to the same $\ell$-sequence in $G$ (or vertex in  $S^{\ell}(G)$).
Finally, each adjacency in $L^{\ell}(G^*)$ corresponds to an adjacency in $S^{\ell}(G)$, and digons turn into edges.
\end{proof}

In particular, if $G$ is a graph, then each vertex $\u$ of $L^{\ell}(G^*)$ has one conjugate $\overline{\u}$ (not necessarily different from $\u$) and $S^{\ell}(G)$ turns out to be the $\ell$-sequence graph, as it was defined by Fiol, Yebra and F\`abrega \cite{fyf83}. Otherwise, if $G$ is a digraph, $S^{\ell}(G)$ coincides with the $\ell$-iterated line digraph $L^{\ell}(G)$ studied in Fiol, Yebra and Alegre~\cite{FiYeAl83,FiYeAl84}.

An example of this situation is shown again in Figure \ref{fig:S3G-L3G*}.
%\begin{figure}[t]
%\begin{center}
%\includegraphics[width=8cm]{fig(LG-SG).pdf}
%%\vskip-1cm
%\caption{From a 3-iterated line digraph to the 3-sequence graph.}
%\end{center}
%\end{figure}
Note that, in fact, the above theorem shows the existence of a homomorphism  from iterated line digraphs to sequence mixed graphs:
$L^{\ell}(G^*) \rightarrow S^{\ell}(G)$.
%This fact also has consequences on the spectra of $L^{\ell}(G^*)$ and $S^{\ell}(G)$. More precisely, the non-zero eigenvalues of $L^{\ell}(G^*)$ are also eigenvalues of $S^{\ell}(G)$. For instance, in Figure \ref{fig:S3G-L3G*} the spectrum of (the adjacency matrix of) $L^3(G^*)$ is
%$$
%2^1, 0^{14}, -2^1
%$$
%whereas the spectrum of  $S^3(G)$  is
%$$
%(\sqrt{5}+1)^1, 2^2, (\sqrt{5}-1)^1, 1^4, -1^4, (1-\sqrt{5})^1, -2^2, (-1-\sqrt{5})^1.
%$$

Another consequence of Theorem \ref{th:sequencenc-line-digraph} is the following result.

\begin{corollary}
\label{coro-Sl=Ll}
Let $G$ be a mixed graph with maximum undirected degree $\Delta=1$. Then, for even $\ell\ge 2$, the $\ell$-sequence mixed graph $S^{\ell}(G)$ is isomorphic to the $\ell$-iterated line digraph $L^{\ell}(G)$ provided that the edges of the former are mapped to the digons of the latter.
\end{corollary}
\begin{proof}
As $\Delta=1$, the undirected subgraph of $G$ consists of some complete graphs $K_2$ with two vertices and one edge. Then, the associated digraph $G^*$ of $G$ is obtained by replacing each of such edges by a digon and, since $\ell$ is even, $G^*$ has no conjugate vertices. Besides, for every pair of vertices $u,v$ of a digon, $L^{\ell}(G)$, and also  $S^{\ell}(G)$, has the vertices $uvuv\ldots uv$ and $vuvu\ldots vu$ forming a digon in the former and an edge in the latter. All the other adjacencies of $L^{\ell}(G)$ remain unchanged in $S^{\ell}(G)$. Then, the claimed isomorphism follows.
\end{proof}

\subsection{Distance-related parameters}\label{sec:drp}

As expected, the diameter and average distance of a sequence mixed graph satisfy analogous results to those in Theorem \ref{th:diam-line}.

\begin{theorem}\label{th:diam-mixed}
Let $G$ be a mixed graph with diameter $k$ and average distance $\overline{k}$. Then, the diameter $k_{\ell}$ and  average distance $\overline{k}_{\ell}$ of the $\ell$-sequence mixed graph $S^{\ell}(G)$ satisfy
\begin{align}
k_{\ell} &\le  k+\ell, \\
\overline{k}_{\ell} &< \overline{k}+\ell.\label{eq:dist-mitj}
\end{align}
\end{theorem}
\begin{proof}
Clearly, the associated digraph $G^*$ has the same diameter $k$ and average distance $\overline{k}$ as $G$.
Moreover, by Theorem \ref{th:diam-line}, the diameter  and average distance of the iterated line digraph $L^{\ell}(G^*)$ satisfy $k^*_{\ell}=k+\ell$ and $\overline{k}^*_{\ell}<\overline{k}+\ell$. Consequently, the result follows since the procedure of Theorem \ref{th:sequencenc-line-digraph} to obtain $S^{\ell}(G)$ from $L^{\ell}(G^*)$ (by identifying vertices) clearly does not increase either the diameter or the average distance. That is, $k_{\ell}\le k^*_{\ell}$ and  $\overline{k}_{\ell}\le \overline{k}^*_{\ell}$.
\end{proof}

Note that, in particular, if $G$ is a graph, then (\ref{eq:dist-mitj}) provides a bound for the average distance of its $\ell$-sequence graph, and this extends the result of Theorem~\ref{th:diam} given in Fiol, Yebra and F\`{a}brega~\cite{fyf83}.

A constructive proof of the result on the diameter, which also gives a routing algorithm, is the following:
Given two vertices of $S^{\ell}(G)$, $\u=u_0u_1\dots u_{\ell-1}u_{\ell}$ and  $\vecv=v_1v_2\dots  v_{\ell}v_{\ell+1}$, let us consider a shortest path from $u_{\ell}$ to $v_1$ in $G$, say $u_{\ell},w_1,\ldots,w_{p-1},v_1$, of length $p=\dist(u_{\ell},v_1)\le k$. Then from our adjacency rule in Definition \ref{def:gen-seq-graphs}, we can shift the first $\ell$-sequence (defining $\u$) in
$$
u_0,u_1,\dots, u_{\ell-1},u_{\ell},w_1,\ldots,w_{p-1},v_1,v_2,\dots,  v_{\ell},v_{\ell+1}
$$
to the last one (defining $\vecv$), obtaining a path from $\u$ to $\vecv$ in $S^{\ell}(G)$ of length $p+\ell\le  k+\ell$.

\begin{figure}[t]
\begin{center}
\vskip-.75cm
\includegraphics[width=12cm]{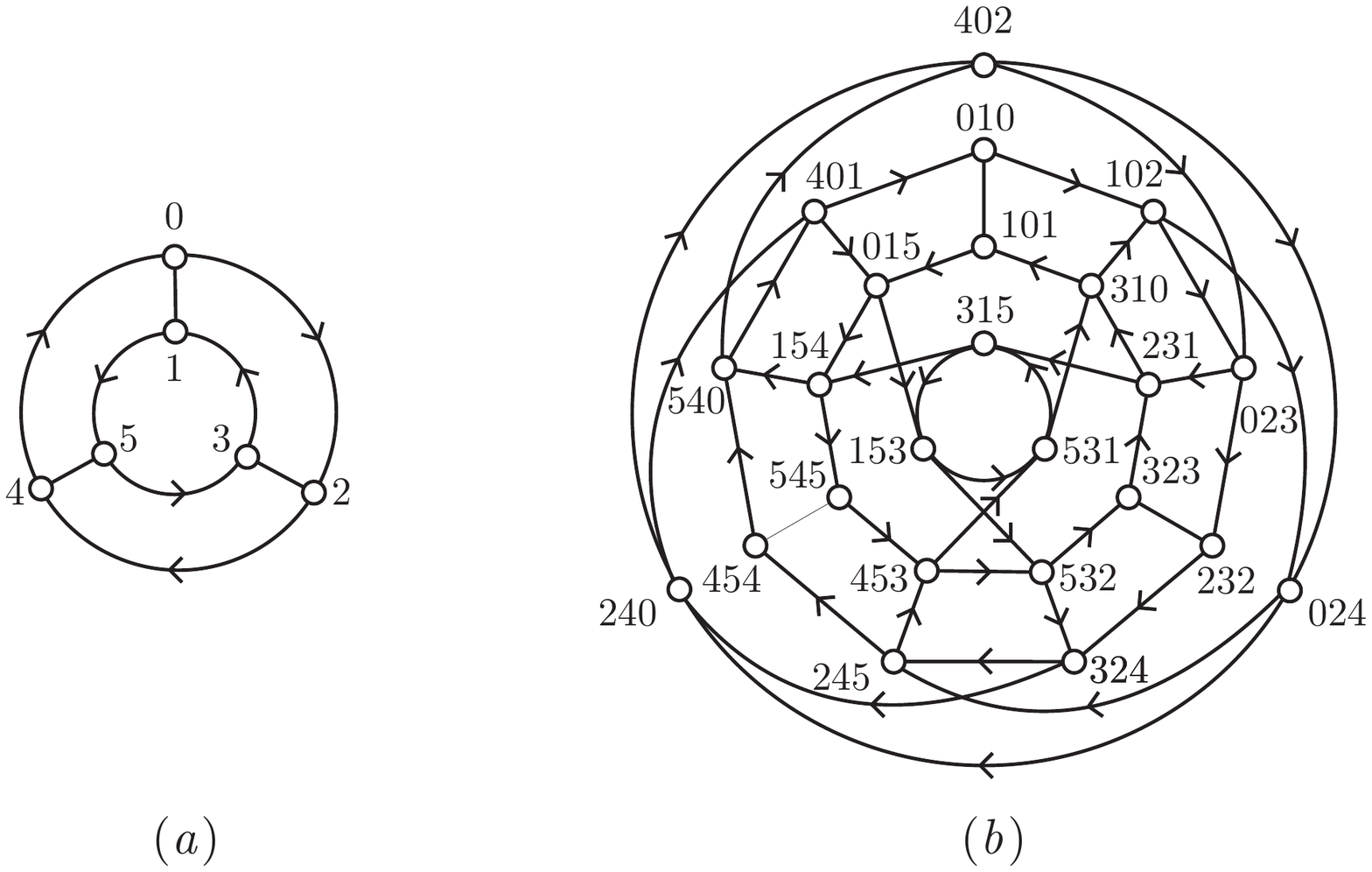}
\vskip-9.75cm
\caption{$(a)$ The Kautz digraph $K(2,2)$ as a mixed graph; $(b)$ Its 2-sequence graph.}
\label{fig:S^2(K(2,2))}
\end{center}
\end{figure}

For example, if $G$ is the Kautz digraph $K(2,2)$ (seen as a mixed graph) on $6$ vertices with degree $2$ and diameter $2$, its sequence mixed graph $S^2(G)$  has diameter $4$. Both mixed graphs are shown in Figure \ref{fig:S^2(K(2,2))}.

\subsection{Sequence mixed graph and the degree/diameter problem}

As it is implicit in the results of Proposition \ref{prop:reg1}, in general the sequence graph operator does not preserve the regularity of a mixed graph.
However,
in the context of the degree/diameter problem, our results lead to the following example of application.
The Kautz digraphs $G=K(\delta,2)$ are mixed Moore graphs of diameter $2$ for $r=1$ and $z=\delta-1$.
From \eqref{eq:moorebound4} with $k=2$, Proposition \ref{prop:reg1} with $\ell=2$, and Theorem \ref{th:diam-mixed}, its $2$-sequence mixed graph  has order
$$
|S^2(G)|=\frac{\delta(\delta+1)}{2}(2\delta^2-1)+\frac{\delta(\delta+1)}{2}=\delta^3(\delta+1),
$$
diameter $4$, undirected degree at most $r=1$ and directed degree at most $z=\delta$. Now, we can compare the order of $S^2(G)$ with the mixed Moore bound $M(1,\delta,4)$ which, according to \eqref{eq:moorebound4}, is
$$
M(1,\delta,4)=\delta^4+5\delta^3+7\delta^2+4\delta+2.
%d^4 + 5*d^3 + 7*d^2 + 4*d + 2
$$
Thus, asymptotically, with total degree $d=1+\delta$, the mixed graph $S^2(G)$ has order equal to the mixed Moore bound:
$$
\lim_{d\rightarrow \infty} \frac{|S^2(G)|}{M(1,\delta,4)}=1.
$$
In fact, notice that, according to Corollary \ref{coro-Sl=Ll}, $S^2(G)$ is isomorphic (up to the equivalence between edges and digons) with the Kautz digraph $K(\delta,4)=L^2(K(\delta,2))$.

In general, although the sequence mixed graph asymptotically provides good results in the context of the degree/diameter problem, it is far from giving large mixed graphs for small values of $r,z$ and $k$. For instance, consider $G$ to be again the referred Bos\'ak graph, with $18$ vertices, $r=3,z=1$ and $k=2$. According to Proposition \ref{prop:reg1}, $S^1(G)$ has maximum undirected degree $4$, maximum directed out-degree $4$ and diameter $3$. Then, the corresponding mixed Moore bound is $M(4,4,3)=521$, although $S^1(G)$ has only $45$ vertices. In order to obtain better results, we give the following proposition.

\begin{proposition}
\label{prop:imp}
Let $G$ be an $(r,z)$-regular mixed graph of order $n$, with $d=r+z$, for $r>1$, and consider its $\ell$-sequence mixed graph $S^{\ell}(G)$. Then, for any given integer $r'\in\{1,\ldots,z\}$ there exists a mixed graph $S^{\ell}_m(G)$ of order $N_{\ell}(G)$ given in Proposition \ref{prop:reg1}$(c)$, diameter $\diam(S^{\ell}_m(G)) \leq \diam(S^{\ell}(G))$, and the following maximum degrees:
\begin{itemize}
\item[$(a)$]
For $\ell=1$,  $S^{1}_m(G)$ has maximum undirected degree $\max\{2r-2,2r'\}$ and maximum directed out-degree $\max\{d-r',2z\}$.
%\item[$(b)$]
%For $\ell>1$, $r>1$ and for some integer $r'=1,\ldots,d-1$,  $S^{\ell}_m(G)$ has maximum undirected degree $\max\{2r,2r'\}$ and maximum directed out-degree $\max\{d-r',2z\}$.
\item[$(b)$]
For $\ell>1$,  $S^{\ell}_m(G)$ has maximum undirected degree $\max\{2r,2r'\}$ and maximum directed out-degree $\max\{d-r',2z\}$.
%\item[$(b'')$]
%For $\ell>1$, $r>1$,  $S^{\ell}_m(G)$ has maximum undirected degree $2(r+z)$ and maximum directed out-degree $\max\{r,2z\}$.
 \end{itemize}
\end{proposition}

\begin{proof}
Let $V_2$ be the set of vertices in $S^{\ell}(G)$ as in Proposition \ref{prop:reg1}, that is, corresponding to $\ell$-walks of $G$ containing some arc.

$(a)$   For $\ell=1$, any vertex $\vecv \in V_2$ has exactly $z$ in-neighbors and $z$ out-neighbors in $V_2$, so the induced subgraph of $S^1(G)$ generated by $V_2$ is a $z$-regular digraph. Then, if we turn all these arcs of $S^1(G)$ into edges, we obtain the mixed graph $S^1_m(G)$  having  the described properties for $r'=z$. When $r' \leq z-1$, by Hall's theorem (see, for example, Brualdi \cite{b10}), the induced subgraph of $S^1(G)$ generated by $V_2$ contains an $r'$-factor (that is, a spanning $r'$-regular digraph). Then, the claimed mixed graph $S^{1}_m(G)$ is obtained, as before, by changing all the arcs of this $r'$-factor by edges.

%$(b)$ When $\ell,r>1$, the induced subgraph of $S^{\ell}(G)$ generated by $V_2$ is a $d$-regular digraph (NO ES VERITAT, en general, molts vertexs $v$ de $V_2$ tenen grau d'entrada i sortida $d$ en $V_2$, però per exemple els vèrtexs corresponents a una sequencia $\vecv=v_1v_2\dots  v_{\ell}v_{\ell+1}$ on $v_1v_2$ és una arc i la resta $v_iv_{i+1}$ son arestes tenen grau d'entrada $d$ a $V_2$, però grau de soritda $z$ a $V_2$ (els restants $r$ arcs de sortida van a parar a $V_1$).% Thus, by applying Hall's theorem, it contains an $r'$-factor (that is, a spanning $r'$-regular digraph) with $1\le r'\le d-1$ Then,  the claimed mixed graph  $S^{\ell}_m(G)$ is obtained, as before, by changing all the arcs of such a $r'$-factor by edges.

$(b)$ The induced subgraph of $S^{\ell}(G)$ generated by $V_2$ is no longer a regular digraph when $\ell >1$, but it contains a $z$-regular subdigraph. Indeed, every vertex  $u_0u_1\dots u_{\ell-1}u_{\ell}\in V_2$ is adjacent from (respectively, to) the $z$ vertices of the form $vu_0u_1\dots u_{\ell-1}$ (respectively, $u_1\dots u_{\ell-1}u_{\ell}w$) where $(v,u_0)$ (respectively, $(u_{\ell},w)$) is an arc of $G$. Now we reason as in $(a)$ to obtain the result.
%$(b'')$ When $\ell,r>1$, we turn all the arcs of the induced subgraph of $S^{\ell}(G)$ generated by $V_2$ into edges and the claimed mixed graph $S^{\ell}_m(G)$ is obtained.
\end{proof}

Proposition \ref{prop:imp} is useful when $z \leq r$. In this case, $S^{\ell}_m(G)$ preserves the maximum undirected degree of $S^{\ell}(G)$, meanwhile the maximum directed degree is reduced by $r'$, that is, from $d$ to $d-r'$. This means that the order of $S^{\ell}_m(G)$ has to be compared with the mixed Moore bound $M(2r,d-r',k')$ instead of $M(2r,d,k)$, for $k'\leq k$, which is an improvement. Going back to the example of the Bos\'ak graph $G$, now $S^1_m(G)$ has the same order, diameter and undirected degree than $S^1(G)$, but the directed out-degree becomes now $3$ instead of $4$, so as the mixed Moore bound reduces from $521$ to $M(4,3,3)=344$. Hence, we improve the ratio between the order of the mixed graph and the corresponding mixed Moore bound. Note that, in the above modification of the sequence graph of the Bos\'ak graph, we are not reducing the diameter. %In contrast, by taking some circulant graphs as a basis, we find out that $S^{\ell}_m(G)$ allows us to halve the diameter of $S^{\ell}(G)$. For instance, we could take the circulant graphs with order $N=2k^2+2k+1$ and set of `steps' $\Lambda=\{\pm k,k+1\}$, which correspond to the Cayley graph $\Cay(\Z_{N},\Lambda)$, and it has the maximum order for given diameter in such context of Cayley graphs of cyclic groups. For more details, see
%Yebra, Fiol, Morillo, and Alegre \cite{yfma85}. The diameters of these modified sequence graphs of circulant graphs do not provide good enough results in the general context
%of the degree/diameter problem.
Thus, an interesting research line would be to give different modifications of $S^{\ell}(G)$ in order to get better
results in the degree/diameter problem. For instance, if one were able to lower the diameter of $S^{1}_m(G)$ from $3$ to $2$, just by
replacing some arcs by edges, or adding some arcs or edges, without modifying the maximum undirected and directed degrees, then $M(4,3,2)=53$ and, hence, the order of this
new mixed graph would be very close to the mixed Moore bound.
%$r \geq \max\{z+1,2z\}$ and $r' \leq \min\{2r-2,z\}$. In this case, $S^1_m(G)$ preserves the maximum undirected degree  of $S^1(G)$ meanwhile the maximum directed degree is reduced by $z$, that is, from $r+z$ to $r$. This means that the order of $S^1_m(G)$ has to be compared with the Moore bound $M(2r-2,r,k')$ instead of $M(2r-2,r+z,k)$, for $k'\leq k$, which is an improvement. An analogous argument works for $\ell >1$ with  $r \geq \max\{z+1,2z\}$ and $r' \leq \min\{2r,z\}$.

%The result in case $(a)$ is useful when $r \geq \max\{z+1,2z\}$. In this case, $S^1_m(G)$ preserves the maximum undirected degree $2r-2$ of $S^1(G)$ meanwhile the maximum directed degree is reduced by $z$, that is, from $r+z$ to $r$. This means that the order of $S^1_m(G)$ has to be compared with the Moore bound $M(2r-2,r,k')$ instead of $M(2r-2,r+z,k)$, for $k'\leq k$, which is an improvement.  %In particular, if $r=2$ and $z=1$, then $S^1_m(G)$ is a $(2,2)$-regular mixed graph.

Note that the construction of case of Proposition \ref{prop:imp}$(b)$ can be also applied when $G$ is a digraph ($r=0$). Then, the resulting mixed graph has maximum undirected degree $2r'$ and maximum directed degree $d-r'$. For instance, we know that, if $G=K_{d+1}^*$ (the complete symmetric digraph on $d+1$ vertices), $S^{\ell}(G)=L^{\ell}(G)$ is the Kautz digraph $K(d,\ell)$. Then, for $1\le r'\le d-1$, we obtain the mixed graph $S^{\ell}_m(G)$ with number of vertices $d^{\ell}+d^{\ell-1}$, diameter $\ell$, and the above maximum degrees $2r'$ and $d-r'$. Then, the ratio
$
M(2r',d-r',\ell)/|S^{\ell}_m(G)|
$
turns out to be of the order of $(1+r'/d)^{\ell}$, which is asymptotically optimal for a fixed $r'$ and large $d$.

\vskip.5cm

\noindent{\bf Acknowledgments.}  This research is supported by the
{\em Ministerio de Ciencia e Innovaci\'on}, and the {\em European Regional Development Fund} under project MTM2014-60127-P (C. D. and M. A. F.), the {\em Catalan Research Council} under project 2014SGR1147 (C. D. and M. A. F.). The author N. L. has been supported in part by grant MTM2013-46949-P, from {\em Ministerio de Econom\'{\i}a y Competitividad}, Spain.

%
%\newpage
%%%%%%%%%%%%%%%%%%%%%%%%%%%%%%%%%%%%%%%%%%%%%%%%
%Bibliografia
%%%%%%%%%%%%%%%%%%%%%%%%%%%%%%%%%%%%%%%%%%%%%%%%

\end{document}